\documentclass[11pt,a4paper]{amsart}

\usepackage[T1]{fontenc}
\usepackage{lmodern}
\usepackage{microtype}
\usepackage{amsmath,amssymb,amsthm,mathtools}
\usepackage[a4paper,margin=1.12in]{geometry}
\usepackage[hidelinks]{hyperref}
\hypersetup{
  pdftitle={Cofinal towers of hyperbolic link complements with vanishing homology torsion},
  pdfauthor={Qilong Guo},
  pdfsubject={Hyperbolic 3-manifolds and homology torsion},
  pdfkeywords={hyperbolic 3-manifold, cofinal tower, homology torsion, right-angled polyhedron, link complement}
}

\newcommand{\HH}{\mathbb H}
\newcommand{\ZZ}{\mathbb Z}
\newcommand{\cF}{\mathcal F}
\newcommand{\Tor}{\operatorname{Tor}}
\newcommand{\vol}{\operatorname{vol}}
\newcommand{\Isom}{\operatorname{Isom}}

\theoremstyle{plain}
\newtheorem{theorem}{Theorem}[section]
\newtheorem{proposition}[theorem]{Proposition}
\newtheorem{corollary}[theorem]{Corollary}

\theoremstyle{remark}
\newtheorem{remark}[theorem]{Remark}

\title[Cofinal towers with vanishing homology torsion]
{Cofinal towers of hyperbolic link complements\\with vanishing homology torsion}

\author{Qilong Guo}
\address{College of Science, China University of Petroleum--Beijing,
Beijing 102249, P. R. China}
\email{guoqilong1984@hotmail.com}



\begin{document}

\begin{abstract}
Problem~3.6 in the $\mathrm{K3}$ problem list of Baykur, Kirby and Ruberman asks whether every cofinal tower
\[
  M_0\longleftarrow M_1\longleftarrow M_2\longleftarrow\cdots
\]
of finite covers of a finite-volume hyperbolic $3$-manifold satisfies
\[
  \lim_{n\to\infty}
  \frac{\log|\Tor H_1(M_n;\ZZ)|}{\vol(M_n)}
  =\frac{1}{6\pi}.
\]
We give a negative answer.  For every ideal right-angled polyhedron $P_0$, the checkerboard manifold associated to $P_0$ admits a cofinal tower all of whose levels are hyperbolic link complements in $S^3$.  Thus $\Tor H_1(M_n;\ZZ)=0$ at every level, and the normalized logarithmic homology torsion is identically zero.
\end{abstract}

\maketitle

\section{Introduction}

Problem~3.6 in the $\mathrm{K3}$ problem list of Baykur, Kirby and Ruberman~\cite{BKR} asks whether every cofinal tower of finite covers
\[
  M_0 \longleftarrow M_1 \longleftarrow M_2 \longleftarrow \cdots
\]
of a finite-volume hyperbolic $3$-manifold satisfies
\begin{equation}\label{eq:K3-limit}
  \lim_{n\to\infty}
  \frac{\log |\Tor H_1(M_n;\ZZ)|}{\vol(M_n)}
  =\frac{1}{6\pi}.
\end{equation}
Here cofinal means that, after choosing compatible basepoints, the corresponding nested subgroups of $\pi_1(M_0)$ have trivial intersection.  We show that cofinality alone does not force~\eqref{eq:K3-limit}.

The constant $1/(6\pi)$ comes from the $L^2$-torsion of hyperbolic $3$-space.  Bergeron and Venkatesh conjectured this limit for congruence towers of closed arithmetic hyperbolic $3$-manifolds~\cite{BV}.  L\^e proved the corresponding upper bound for exhaustive nested normal towers.  His result applies more generally to trace-convergent sequences of covers of irreducible $3$-manifolds with empty or toroidal boundary~\cite{Le}.

Computations of \c{S}eng\"un~\cite{SengunBianchi,SengunTetrahedral} and the experiments of Brock and Dunfield~\cite[Section~4]{BD} support the value $1/(6\pi)$ in several arithmetic and nonarithmetic settings.  Brock and Dunfield also produced a Benjamini--Schramm convergent sequence of closed hyperbolic $3$-manifolds for which normalized analytic torsion does not converge to $1/(6\pi)$~\cite{BD}; their sequence is not a tower over a fixed manifold.  Champanerkar and Kofman constructed noncofinal towers of cusped arithmetic manifolds with torsion-growth limit $1/(4\pi)$~\cite{CK}.  The role of the free part of homology and its regulator is discussed in~\cite{BSV}.

Our examples form cofinal towers over a fixed cusped manifold, and every level is a link complement.

\begin{theorem}\label{thm:main}
Let $P_0\subset\HH^3$ be an ideal right-angled polyhedron, and let $M_0=N(P_0)$ be its checkerboard manifold.  There is a cofinal tower
\begin{equation}\label{eq:main-tower}
  M_0 \longleftarrow M_1 \longleftarrow M_2 \longleftarrow \cdots
\end{equation}
such that every bonding map $M_{n+1}\to M_n$ is a regular double cover and every $M_n$ is a hyperbolic link complement in $S^3$.  Moreover,
\[
  \deg(M_n\to M_0)=2^n,
  \qquad
  \Tor H_1(M_n;\ZZ)=0
\]
for every $n$.  In particular,
\[
  \frac{\log |\Tor H_1(M_n;\ZZ)|}{\vol(M_n)}=0
\]
at every level.
\end{theorem}

The argument combines Gir\~ao's cofinal sequences of reflection groups obtained by face doubling~\cite{Girao} with Erokhovets's realization of a four-copy checkerboard manifold as a hyperbolic link complement~\cite{Erokhovets}.  Section~\ref{sec:checkerboard} gives the checkerboard construction in reflection-group language.  Section~\ref{sec:doubling} proves that it is compatible with face doubling, and Section~\ref{sec:tower} applies this compatibility to Gir\~ao's sequence.

\section{Checkerboard manifolds}\label{sec:checkerboard}

We use Erokhovets's four-copy construction in a form suited to
reflection doubling.  Let $P\subset\HH^3$ be an ideal right-angled
polyhedron.  The link of each ideal vertex is a Euclidean right-angled
polygon, hence a quadrilateral.  Thus the $1$-skeleton of $P$ is a
plane $4$-valent graph, and its faces admit a checkerboard coloring
\begin{equation}\label{eq:checkerboard-coloring}
  \cF(P)=\cF_B(P)\sqcup\cF_W(P),
\end{equation}
with adjacent faces of different colors.

For a face $F$ of $P$, let $r_F$ denote reflection in its supporting
plane.  By the Poincar\'e polyhedron theorem,
\[
  W(P)=\langle r_F:F\in\cF(P)\rangle<\Isom(\HH^3)
\]
is discrete and has $P$ as a fundamental chamber.  Its Coxeter
presentation is
\begin{equation}\label{eq:Coxeter-presentation}
  W(P)=\left\langle
  r_F\ (F\in\cF(P))
  \ \middle|\
  r_F^2=1,\ [r_F,r_G]=1\text{ if }F\cap G\text{ is an edge}
  \right\rangle;
\end{equation}
see~\cite[Chapter~6]{Davis}.

We shall also use the following standard chamber--index fact.  If
$H\leq G$ are reflection groups and a fundamental chamber for $H$ is
the union of $k$ $G$-chambers, then
\[
  [G:H]=k.
\]
Indeed, the $H$-orbits of the $G$-chambers are naturally indexed by
the cosets in $H\backslash G$; see
\cite[Section~5.1 and Theorem~6.4.3]{Davis}.

Write
\[
  (\ZZ/2\ZZ)^2=\{(0,0),(1,0),(0,1),(1,1)\}
\]
additively, and set
\[
  e_B=(1,0),\qquad e_W=(0,1).
\]
Define
\begin{equation}\label{eq:checkerboard-map}
  \epsilon_P:W(P)\longrightarrow(\ZZ/2\ZZ)^2,
  \qquad
  \epsilon_P(r_F)=
  \begin{cases}
    e_B,&F\in\cF_B(P),\\
    e_W,&F\in\cF_W(P).
  \end{cases}
\end{equation}
The Coxeter relations are respected, and both colors occur, so
$\epsilon_P$ is surjective.  Put
\begin{equation}\label{eq:checkerboard-kernel}
  \Gamma(P)=\ker\epsilon_P.
\end{equation}

Let $x$ lie in the relative interior of an edge $F\cap G$.  Exactly
four $W(P)$-chambers meet at $x$, and
\[
  \operatorname{Stab}_{W(P)}(x)
  =\langle r_F,r_G\rangle
  \cong(\ZZ/2\ZZ)^2.
\]
Since $F$ and $G$ have different colors, $r_F$ and $r_G$ map to the
two standard basis vectors.  Hence the restriction of $\epsilon_P$
to this stabilizer is an isomorphism.

\begin{proposition}\label{prop:kernel-manifold}
The group $\Gamma(P)$ is discrete, torsion-free, and orientation
preserving, and
\[
  [W(P):\Gamma(P)]=4.
\]
Consequently,
\begin{equation}\label{eq:NP}
  N(P):=\HH^3/\Gamma(P)
\end{equation}
is an orientable complete finite-volume hyperbolic $3$-manifold.
\end{proposition}

\begin{proof}
Surjectivity of $\epsilon_P$ gives
\[
  W(P)/\Gamma(P)\cong(\ZZ/2\ZZ)^2,
  \qquad
  [W(P):\Gamma(P)]=4.
\]
The subgroup $\Gamma(P)$ is discrete because $W(P)$ is discrete.

Let $\omega$ be the orientation character of $W(P)$, and set
$\sigma(a,b)=a+b$.  The homomorphisms $\omega$ and
$\sigma\circ\epsilon_P$ agree on every face reflection, and hence
are equal.  Thus $\Gamma(P)\subset\ker\omega$.

Suppose that $1\neq g\in\Gamma(P)$ has finite order.  Then $g$ is
elliptic and fixes a point $x\in\HH^3$.  Choose $w\in W(P)$ such that
$w(x)\in\overline P$.  Since $\Gamma(P)\triangleleft W(P)$, the
conjugate $wgw^{-1}$ still lies in $\Gamma(P)$; replacing $g$ by this
conjugate, we may assume that $x\in\overline P$.

The stabilizer of an interior point of $P$ is trivial.  A point in
the relative interior of a face $F$ has stabilizer $\langle r_F\rangle$,
which meets $\Gamma(P)$ trivially.  At an edge, the stabilizer is one
of the groups considered above, and the restriction of $\epsilon_P$
to it is injective.  Since all vertices of $P$ are ideal, these exhaust
the possibilities.  This contradicts $g\neq1$, and therefore
$\Gamma(P)$ is torsion-free.

Finally, $\Gamma(P)$ has finite index in the finite-covolume group
$W(P)$.  Hence $N(P)$ is complete and has finite volume; the preceding
arguments show that it is an orientable manifold.
\end{proof}

The quotient $N(P)$ also has a concrete four-copy description.  Label
four abstract copies of $P$ by the elements of $(\ZZ/2\ZZ)^2$.  Across
a face of color $c$, identify the copy labelled $a$ with the copy
labelled $a+e_c$.  Thus
\begin{equation}\label{eq:four-copy}
  N(P)\cong P\times(\ZZ/2\ZZ)^2/\!\sim,
\end{equation}
where
\[
  (p,a)\sim(p,a+e_c)
\]
whenever $p$ lies on a face of color $c$.

The quotient in~\eqref{eq:four-copy} is Erokhovets's checkerboard
manifold~\cite[Construction~5.4]{Erokhovets}.  His equivalence relation
identifies $(p,a)$ and $(p,b)$ whenever $a-b$ lies in the span of the
labels of the faces containing $p$; for the checkerboard labeling this
is precisely the relation generated above, up to interchanging the two
basis vectors.

By~\cite[Proposition~5.1]{Erokhovets}, the $1$-skeleton of $P$ admits a
nonselfcrossing Eulerian cycle $\gamma$.  The associated hyperbolic link
$C_\gamma\subset S^3$ satisfies
\begin{equation}\label{eq:Erokhovets-homeomorphism}
  S^3\setminus C_\gamma\cong N(P)
\end{equation}
by~\cite[Proposition~5.5]{Erokhovets}.  If $C_\gamma$ has $c$ components,
Alexander duality yields
\[
  H_1(N(P);\ZZ)\cong\ZZ^c,
\]
and hence
\begin{equation}\label{eq:H1-torsion-free}
  \Tor H_1(N(P);\ZZ)=0.
\end{equation}

\section{Reflection doubling}\label{sec:doubling}
Let \(F\) be a face of \(P\), and put \(s=r_F\).  The
\emph{reflection double} of \(P\) across \(F\) is
\begin{equation}\label{eq:reflection-double}
  P'=P\cup s(P).
\end{equation}
Since every face adjacent to \(F\) meets it orthogonally, the two
copies fit together to form another ideal right-angled Coxeter
polyhedron.  The common face \(F\) lies in the interior of \(P'\)
and is no longer a boundary face.

Let \(G\neq F\) be a face of \(P\).  If \(G\) is adjacent to \(F\),
then \(G\) and \(s(G)\) lie in the same supporting plane and form
a single face of \(P'\); the reflection in this face is \(r_G\).
If \(G\) is not adjacent to \(F\), then \(G\) and \(s(G)\) remain
distinct faces of \(P'\), with reflections
\[
  r_G
  \qquad\text{and}\qquad
  r_{s(G)}=s r_G s,
\]
respectively.  In either case, the new faces inherit the color of
\(G\).

It follows that every face reflection of \(P'\) belongs to
\(W(P)\), and hence
\[
  W(P')\leq W(P).
\]
Moreover, \(P'\) is a fundamental chamber for \(W(P')\), while
\(P'=P\cup s(P)\) is the union of exactly two \(W(P)\)-chambers.
The chamber--index correspondence therefore gives
\begin{equation}\label{eq:W-index-two}
  [W(P):W(P')]=2.
\end{equation}

The inherited coloring defines a homomorphism
\[
  \epsilon_{P'}:W(P')\longrightarrow(\ZZ/2\ZZ)^2
\]
by the same rule as~\eqref{eq:checkerboard-map}.

\begin{proposition}\label{prop:compatibility}
Under the inclusion $W(P')<W(P)$,
\begin{equation}\label{eq:compatibility}
  \epsilon_{P'}=\left.\epsilon_P\right|_{W(P')}.
\end{equation}
Consequently,
\begin{equation}\label{eq:kernel-pullback}
  \Gamma(P')=W(P')\cap\Gamma(P).
\end{equation}
\end{proposition}

\begin{proof}
It is enough to compare the two maps on the face reflections of $P'$.  Let $G\neq F$ be a face of $P$.  If $G$ is adjacent to $F$, then $sr_Gs=r_G$, and the merged face has the color of $G$.  If $G$ is not adjacent to $F$, then $G$ and $s(G)$ have reflections $r_G$ and $sr_Gs$, respectively, and the same color.  Moreover,
\[
  \epsilon_P(sr_Gs)
  =\epsilon_P(s)+\epsilon_P(r_G)+\epsilon_P(s)
  =\epsilon_P(r_G).
\]
Thus the two maps agree on the face-reflection generators of $W(P')$, proving~\eqref{eq:compatibility}.  Taking kernels gives~\eqref{eq:kernel-pullback}.
\end{proof}

\begin{corollary}\label{cor:double-cover}
Reflection doubling induces a regular double cover
\begin{equation}\label{eq:double-cover}
  N(P')\longrightarrow N(P).
\end{equation}
\end{corollary}

\begin{proof}
By~\eqref{eq:kernel-pullback}, $\Gamma(P')\leq\Gamma(P)$, and
\[
  [\Gamma(P):\Gamma(P')]
  =\frac{[W(P):W(P')][W(P'):\Gamma(P')]}
         {[W(P):\Gamma(P)]}
  =\frac{2\cdot4}{4}=2.
\]
Hence $\Gamma(P')\triangleleft\Gamma(P)$, and the induced cover is regular.
\end{proof}

\section{The cofinal tower}\label{sec:tower}

We use the following reformulation of Gir\~ao's construction~\cite[Section~5, Theorem~5.1]{Girao}.

\begin{theorem}[Gir\~ao]\label{thm:Girao}
For every ideal right-angled polyhedron $P_0$, there is a sequence
\begin{equation}\label{eq:polyhedron-sequence}
  P_0,P_1,P_2,\ldots
\end{equation}
in which $P_{n+1}$ is a reflection double of $P_n$ and
\begin{equation}\label{eq:W-cofinal}
  \bigcap_{n\geq0}W(P_n)=\{1\}.
\end{equation}
\end{theorem}

Gir\~ao uses a double indexing; relabelling the polyhedra in the order in which the successive doublings are carried out gives the sequence above.

\begin{proof}[Proof of Theorem~\ref{thm:main}]
Choose $P_n$ as in Theorem~\ref{thm:Girao}, and set
\[
  W_n=W(P_n),
  \qquad
  \Gamma_n=\Gamma(P_n),
  \qquad
  M_n=N(P_n).
\]
Proposition~\ref{prop:compatibility} and Corollary~\ref{cor:double-cover} give
\[
  \Gamma_{n+1}=W_{n+1}\cap\Gamma_n,
  \qquad
  [\Gamma_n:\Gamma_{n+1}]=2.
\]
Thus $M_{n+1}\to M_n$ is a regular double cover and
\begin{equation}\label{eq:degree}
  \deg(M_n\to M_0)=2^n.
\end{equation}

Since $\Gamma_n\subset W_n$,
\[
  \bigcap_{n\geq0}\Gamma_n
  \subset
  \bigcap_{n\geq0}W_n
  =\{1\},
\]
so the tower is cofinal.

By~\eqref{eq:H1-torsion-free},
\[
  \Tor H_1(M_n;\ZZ)=0
\]
for every $n$, and therefore
\[
  \frac{\log |\Tor H_1(M_n;\ZZ)|}{\vol(M_n)}=0.
\]

\end{proof}

\begin{remark}\label{rem:scope}
The manifolds in the tower are cusped and have positive first Betti number.  Each bonding map is regular, but the composite covers $M_n\to M_0$ are not asserted to be regular.  Thus the construction does not address closed arithmetic congruence towers or towers that are regular over the base and have vanishing first Betti number.  Gir\~ao proves that the reflection groups in his tower cannot all be normal in the initial reflection group; see~\cite[Section~6]{Girao}.
\end{remark}

\end{document}